\documentclass[11pt]{amsart}
\usepackage{amssymb, latexsym, amsmath, amsfonts, hyperref, enumitem, fancyhdr}
\usepackage{amsmath,amscd}
\usepackage{cleveref}
\usepackage{mathrsfs}
\newtheorem{thm}{Theorem}[section]

\newtheorem{lem}[thm]{Lemma}

\theoremstyle{definition}
\newtheorem{defn}[thm]{Definition}
\theoremstyle{remark}

\newtheorem{rem}[thm]{Remark}
\numberwithin{equation}{section}

\hypersetup{
	colorlinks,
	citecolor=black,
	linkcolor=black,
	urlcolor=blue}
\theoremstyle{plain}

	{\end{enumerate}}

\hypersetup{
	colorlinks,
	citecolor=blue,
	linkcolor=blue,
	urlcolor=blue}

\begin{document}
	\title{On Isometric Embedding $\ell_p^m\to S_\infty$ and Unique operator space structure}

	{\let\thefootnote\relax}
	%
	\author{Samya Kumar Ray}
	
	
	\address{Samya Kumar Ray: School of Mathematics and Statistics,Wuhan University, Wuhan-430072, China}
	\email{samyaray7777@gmail.com}
	
	%

	\pagestyle{headings}
\begin{abstract}
We study existence of linear isometric embedding of $\ell_p^m$ into $S_\infty,$ for $1\leq p< \infty$ and unique operator space structure on two dimensional Banach spaces. For $p\in(2,\infty)\cup\{1\},$ we show that indeed $\ell_p^2$ does not embed isometrically into $S_\infty$. This verifies a guess of Pisier and broadly generalizes the main result of \cite{GUR18}. We also show that $S_1^m$ does not embed isometrically into $S_p^n$ for all $1<p<\infty$ and $m\geq 2$. As a consequence, we establish noncommutative analogue of some of the results in \cite{LYS04}. We also show that $(\mathbb{C}^2,\|.\|_{B_{p,q}})$ does not embed isometrically into $S_\infty$ for $2<p,q<\infty.$ The main ingredients in our proofs are notions of Birkhoff-James orthogonality and norm parallelism for operators on Hilbert spaces. These enable us to deploy `infinite descent' type of arguments to obtain contradictions. Our approach is new even in the commutative case. We prove that $(\mathbb{C}^2,\|.\|_{B_{p,q}})$ does not have unique operator space structure whenever $(p,q)\in(1,\infty)\times[1,\infty)\cup[1,\infty)\times(1,\infty)$ by showing that they do not have Property P or two summing property. In view of \cite{MIPV19}, this produces genuinely new examples of two dimensional Banach spaces without unique operator space structure, providing a partial answer to a question of Paulsen. In this case, we derive our result by transferring the problem to real case and applying known results of \cite{ARFJS95}.
	\end{abstract}
	
	\maketitle
	\section{Introduction and main results}
	In this article, we investigate existence of linear isometric embedding of $\ell_p^m$ into $S_\infty$, the set of all compact operators on $\ell_2,$ for $1\leq p<\infty,$ and find new examples of two dimensional Banach spaces $(\mathbb{C}^2,\|.\|_{B_{p,q}})$ failing to have unique operator space structure for $(p,q)\in(1,\infty)\times[1,\infty)\cup[1,\infty)\times(1,\infty),$ where $B_{p,q}:=\{(z_1,z_2):|z_1|^p+|z_2|^q<1\}$ is an open subset of $\mathbb C^2$ which is indeed an open unit ball with respect to some norm on $\mathbb C^2,$ denoted by $\|.\|_{B_{p,q}}.$ Our motivation behind investigating this problem is twofold. In the field of functional analysis, one of the important topics is isometric theory of Banach spaces. This line of research actually starts with the pioneering work of Banach himself \cite{BA32}, where he characterized linear isometries of $\ell_p$ and $L_p$ spaces for $1\leq p<\infty.$ This work was further taken up by Lamperti \cite{LA58}. We mention \cite{FLJ03}, \cite{FLJ07} and references therein for a comprehensive study of isometries on Banach spaces. The isomorphic theory of Banach spaces is also an extremely important topic and in last fifty years lot of tools have been developed for this purpose. We refer \cite{JOL01} and \cite{PI86} for many topics around this. However, isometries between $\ell_p$ and $\ell_q$ for $1\leq p\neq q\leq \infty,$ was first studied in \cite{LYV93}. In the article \cite{LYS04}, the authors proved a remarkable result, which is, for $2\leq m\leq n,$ if there exists a linear isometry from $\ell_p^m$ to $\ell_q^n,$ where $1\leq p\neq q\leq \infty,$ then we must have $p=2$ and $q$ an even integer. This topic is also closely related to Warning's problem, cubature formulae and spherical design. We refer \cite{LY08}, \cite{LY09}, \cite{LYS01}, \cite{LYV93} and \cite{LYS04} for more information in this direction.

Motivated by mathematical physics and quantum mechanics, noncommutative mathematics have been developed rapidly in past few decades. Schatten-$p$ class is one of the simplest noncommutative $L_p$-spaces and they can be thought of noncommutative analogue of $\ell_p$-spaces. Isometries on Schatten-$p$ class was first studied by Arazy \cite{AR75}. Isometries on noncommutative $L_p$-spaces associated with semifinite von Neumann algebra was studied in \cite{YE81}. Isometries on noncommutative $L_p$-spaces for general von Neumann algebras and also complete isometries have been studied by \cite{JURS05}. Therefore, it is tempting to examine existence of isometric embedding of Schatten-$p$ class into Schatten-$q$ class for $1\leq p\neq q\leq\infty$ and investigate a noncommutative analogue of work done in \cite{LYS04}.

	Our second motivation stems from operator space theory and its connection with boundary-normal dilation. Suppose $\mathcal{D}\subseteq\mathbb{C}^m$ is a bounded domain which is an open unit ball with respect to some norm and $(M_{n},\|.\|_{op})$ denotes the Banach space of $n\times n$ complex matrices endowed with the usual operator norm. For any $w\in\mathcal{D},$ and $m$-tuple of matrices $\textbf{T}:=(T_1,...,T_m),$ $T_i\in M_n$, for $1\leq i\leq m,$ define $\langle\Delta f(w),\textbf{T}\rangle:=\:\sum_{j=1}^m\frac{\partial f}{\partial z_j}(w)T_j,$ where $ f\in\mathcal{O}(\overline{\mathcal{D}})$, and $\mathcal{O}(\overline{\mathcal{D}})$ denotes closure of all polynomials on $\overline{\mathcal{D}}$ with respect to supremum norm. The map defined as
	\begin{equation}\label{PARHOM} 
	\Phi_{(w,\textbf{A})}(f):=
	\begin{pmatrix}
	f(w)I_n & \langle\Delta f(w),\textbf{T}\rangle\\
	0 & f(w)I_n\\
	\end{pmatrix},
	\end{equation} for all $f$ in $\mathcal{O}(\overline{\mathcal{D}}),$ is clearly a continuous unital algebra homomorphism from $(\mathcal{O}(\overline{\mathcal{D}}),\| \cdot \|_{\infty})$ to $(M_{2n}, \| \cdot \|_{op})$. We call such a homomorphism a {\it Parrot Like Homomorphism}. The origin of this kind of homomorphism is from \cite{PA70} and then further studied in \cite{MI90}--\cite{MIS190}. It is an intriguing open problem (see \cite[Chapter 7., Chapter 14. Page 190]{PA02}, \cite[Chapter 4, Page 93]{PI13}) to determine domains for which there exists a contractive homomorphism  which is not completely contractive. The issue is related to existence of Arveson's notion of boundary normal dilation (see \cite{PA02}). It follows that for $\mathcal D\subseteq \mathbb C^m$ every contractive Parrot like homomorphism is completely contractive if and only if the associated Banach space $(\mathbb C^m,\|.\|_{\mathcal D})$ has unique operator space structure. Paulsen \cite{PA92} used deep facts of Banach space geometry to show that any Banach space with dimension $\geq 5$, can be endowed with different operator space structures. The same is true due to E. Ricard \cite[Chapter 3, Page 79]{PI03} for dimension $\geq 3.$ However, for unit balls of two dimensional complex Banach spaces, the issue is not completely resolved. Till now $\ell_1^2$ and $\ell_\infty^2$ are the only known Banach spaces with unique operator space structure. We refer \cite[Chapter 7, Chapter 14.]{PA02}, \cite[Chapter 4.]{PI13} and \cite[Capter 3.]{PI03} for a detailed discussion on these topics. Recently, in \cite{MIPV19}, the authors have answered this question for two dimensional Banach spaces which embeds isometrically into $S_\infty^2,$ by intrinsically showing that this class of Banach spaces do not have Property P unless it is isometric to $\ell_\infty^2.$ Property P serves as a good testing condition for unique operator space structure and is known to be equivalent to two summing property (see \cite{ARFJS95}, \cite{BAM95}), which was originally observed by Pisier. It is known that $\ell_p^2$ does not have Property P for $1<p<\infty$ (see \cite{BAM95}). The methods from \cite{MIPV19} prompted the question that which finite dimensional Banach spaces can be embedded isometrically into $S_\infty^n$ for some $n\geq 1$. It was indeed shown in \cite{GUR18} that $\ell_1^m$ does not embed isometrically into $S_\infty^n$ for any $n\geq m>1.$ After this Pisier guessed \cite[Remark 2.5]{GUR18} that it is probably true that there is no isometric embedding of $\ell_1^m$ into compact operators on $\ell_2$ for $m>1.$ More generally, G. Misra asked whether for $m>1,$ $\ell_p^m$ embeds isometrically into $S_\infty,$ for $1\leq p\neq 2<\infty$ ? However, the methods of \cite{GUR18} fail to answer these questions. Note that if there is no isometric embedding of $\ell_p^m$ into compact operators, then we immediately have $S_p^m$ does not embed isometrically into $S_\infty$ for $1\leq p<\infty.$ This is how our motivations coincide.

One of our main theorems is the following which gives a completely new proof of the main result in \cite{GUR18} and extends it to further generality, verifying the guess of Pisier to be correct.
	\begin{thm}\label{noc}
		There is no linear isometry from $\ell_1^m$ to $S_\infty$ for $m\geq 2.$
	\end{thm}
	We also provide an answer to the question of G. Misra for the range $2<p<\infty.$
	\begin{thm}\label{forp}
		Let $2<p<\infty.$ Let $\Phi:\ell_p^2\to B(\ell_2)$ defined by $\Phi((z_1,z_2)):=z_1T+z_2S$ is a linear isometry. Then, we must have $T,S\in B(\ell_2)\setminus S_\infty.$
	\end{thm}
Note that the above theorem readily implies that there is no isometric embedding of $S_p^m$ into $S_\infty^n$ for $m\geq 2$ and $2<p<\infty.$
We also prove the following theorem.
\begin{thm}\label{ncprecha}
Let $1<p<\infty.$ There is no isometric embedding of $S_1^m$ into $S_p$ for $m\geq 2.$
\end{thm} 
Together with \Cref{noc}, \Cref{ncprecha} and \Cref{forp}, we immediately have that $S_p^m$ does not embed isometrically into $S_q^n$ for all $(p,q)\in(2,\infty)\times\{\infty\}\cup\{1\}\times(1,\infty]$ and $m\geq 2.$ This partially establishes noncommutative analogue of the main theorem in \cite{LYS04} for the above mentioned range of $p,q.$
The main tools which we use, are the notions of Birkhoff-James orthogonality and norm parallelism. Birkhoff-James orthogonality has its origin in \cite{BI35}, \cite{JA45} and \cite{JA47}. A remarkable equivalent criterion for Birkhoff-James orthogonality was established in \cite{BHS99}. These tools enable us to obtain new isometry from existing isometry and finally use classic trick of `infinite descent' (see \cite{EN98} for many other illustrations) to obtain contradiction. Our approach is even new in the commutative case and we believe our methods can be applied to study isometries between other kind of Banach spaces also for example Orlicz spaces and symmetric operator spaces and might be helpful in simplifying existing proofs in \cite{LY08}--\cite{LYS04}.

Now, we state our main results concerning the problem on existence of unique operator space structure. In this paper, we prove:
\begin{thm}\label{auoss}
Let $1\leq p,q<\infty$ and $p$ or $q$ is strictly bigger than $1.$ Then, $(\mathbb C^2,\|.\|_{B_{p,q}})$ does not have Property P and hence unique operator space structure.
\end{thm}
The above theorem provides simpler proofs of \cite[Theorem 6.2]{MIPV19} and \cite[Theorem 2.3]{BAM95}.
In view of \cite{MIPV19}, this produces new examples of two dimensional Banach spaces without unique operator space structure. Our proof mainly relies on a transfer principle of complex to real case which enables us to apply known characterization of Property P established in \cite{ARFJS95}. 

We end our introduction by mentioning the organization of rest of the paper. In Section \ref{sec1}, we describe all the necessary preliminaries and notations. In Section \ref{sec2}, we prove some elementary lemmas. In Section \ref{sec3}, we give the proofs of the main theorems as described in the introduction. 
	
\section{Preliminaries}	\label{sec1}
	\subsection{Birkhoff-James orthogonality and norm-parallelism} Unless specified we always work with complex Banach spaces. Denote $X$ to be a Banach space. Let $B(X)$ denote the set of bounded linear maps on $X.$ For $x,y\in X,$ we say that $x$ is orthogonal to $y$ in the sense of Birkhoff-James if $\|x+zy\|\geq\|x\|$ for all complex number $z\in\mathbb C.$ In this case, we write $x\perp_{BJ}y.$ 
	
	An element $x\in X$ is said to be norm parallel to another element $y\in X,$ if there exists $z\in\mathbb T:=\{z\in\mathbb C:|z|=1\},$ such that $\|x+zy\|=\|x\|+\|y\|.$ In this case, we write $x||y.$
	
	For any $T\in B(X),$ let us denote ${\mathbb M}_T:=\{x\in X:\|x\|=1,\ \|Tx\|=1\}.$

Let $1\leq p\leq\infty$. Let $\ell_p^n$ and $\ell_p$ denote the usual sequence spaces. If the scalar field is $\mathbb R$ we denote these spaces by $\ell_p^n(\mathbb R)$ and $\ell_p(\mathbb R)$ respectively.

For $1\leq p<\infty,$ one defines $S_p:=\{T\in B(\ell_2):Tr(|T|^p)<\infty\},$ where $|T|:=(T^*T)^{\frac{1}{2}}$ and $Tr$ denotes the usual trace. Endowed with the norm $\|T\|_p:=(Tr(|T|^p))^{\frac{1}{p}},$ $S_p$ becomes a Banach space. The space of compact operators on $\ell_2$ is denoted by $S_\infty,$ which is again a Banach space endowed with the norm of $B(\ell_2).$ One defines $S_p^n$ in an analogous way, by replacing $\ell_2$ in the above definitions by $\ell_2^n.$ 

We present the following theorems after paraphrasing.
\begin{thm}\cite{TU17}\label{ort}
	Let $T,S\in B(\ell_2)$ such that $T$ be compact. Then, $T\perp_{BJ}S$ if and only if there exists $\zeta\in \mathbb{M}_T$ such that $\langle T\zeta, S\zeta\rangle=0.$
\end{thm}
\begin{thm}\cite{ZA16}\label{para}
	Let $T,S\in S_\infty$. Then, $T||S$ if and only if there exists $\zeta\in \mathbb{M}_T\cap \mathbb{M}_S$ such that $|\langle T\zeta,S\zeta\rangle|=\|T\|_{\infty}\|S\|_{\infty}.$
	\end{thm}
\begin{thm}\cite{BOCMWZ17}\label{FM1}
	Let $1<p<\infty.$ Let $T,S\in S_p.$ Then, $T||S$ if and only if $T$ and $S$ are linearly dependent. 
\end{thm}

\subsection{Property P and operator space:}	
Recall that a domain $\Omega\subseteq\mathbb C^n$ is said to be Reinhardt if $(z_1,z_2,\dots,z_n)\in\Omega$ implies $(e^{i\theta_1}z_1,e^{i\theta_2}z_2,\dots,e^{i\theta_n}z_n)\in\Omega$ for all real numbers $\theta_1,\theta_2,\dots,\theta_n.$ Given a complex Banach space $(\mathbb{C}^2,\|\cdot\|_{\Omega})$, where the unit ball $\Omega$ is Reinhardt, we define a corresponding real two dimensional Banach space $(\mathbb{R}^2,\|\cdot\|_{|\Omega|})$ with norm defined as $\|(x,y)\|_{|\Omega|}:=\|(x,y)\|_{\Omega},$ for $(x,y)\in\mathbb{R}^2$. The unit ball of $(\mathbb{R}^2,\|\cdot\|_{|\Omega|})$ is denoted by $|\Omega|$.

For any two Banach spaces $E$, $F$ and a linear map $A:E\to F$, the operator norm of $A$ is denoted by $\|A\|_{E\to F}$. Often, we use $\|A\|_{op}$ to denote the operator norm, when the underlying Banach spaces on which $A$ acts are well understood. For a complex matrix $A\geq 0$, we define $\mathbf A^{+}=\begin{pmatrix}
a_{11} & |a_{12}|\\
|a_{12}| & a_{22}\\
\end{pmatrix}$, where $A=\begin{pmatrix}
a_{11} & a_{12}\\
\overline{a_{12}} & a_{22}\\
\end{pmatrix}$. Clearly, $\mathbf A^{+}\geq 0$ if and only if $A\geq 0$. Similarly, for a real matrix $A\geq 0$, we define $\mathbf A^{+}$ as we have defined above.

Let $X$ be a finite dimensional Banach space (real or complex), we associate a numerical constant $\gamma(X)$ as the following.
\begin{equation}\label{-1}
\gamma(X):=\sup\{\langle A,B\rangle:A\geq 0,B\geq 0, \ \|A\|_{X\to X^*},\|B\|_{X^*\to X}\leq 1\},
\end{equation}
where the inner product in \eqref{-1} is the Hilbert-Schimdt inner product by representing $A$ and $B$ as matrices with respect to usual basis.
\begin{defn}[Property P]A real or complex Banach space $X$ is said to have Property P if $\gamma(X)=1$.
\end{defn}
It has been proved in \cite{BAM95} (which was originally observed by Pisier), that Property P is actually equivalent to two summing property as defined in \cite{ARFJS95}. The following remarkable characterization is from \cite{ARFJS95}.
\begin{thm}\cite{ARFJS95}\label{twoo}
Let $X$ be real Banach space with two summing property and $\text{dim}X\geq 2.$ Then, $X$ is isometrically isomorphic to $\ell_\infty^2(\mathbb R).$
\end{thm}
	
An abstract operator space is a norm linear space $X,$ which is endowed with a matricial norm structure $(M_n(X),\|.\|_n)$ for all $n\geq 1,$ satisfying the following conditions.
\begin{itemize}
\item $\|A\oplus B\|_{m+n}=\text{max}\{\|A\|_m,\|B\|_n\},$ for all $A\in M_n(X)$ and $B\in M_m(X).$
\item $\|ACB\|_{n}\leq \|A\|_{\ell_2^m\to\ell_2^n}\|C\|_{m}\|B\|_{\ell_2^n\to\ell_2^m},$ for all $A\in M_{n,m},$ $B\in M_{m,n}$ and $C\in M_m(X).$ 
\end{itemize}Given any Banach space $X,$ we have two natural operator space structures on $X$ known as $MIN$ and $MAX$ structures (see \cite[Chapter 3.]{PI03}). The Banach space $X$ has a unique operator space structure iff $MIN=MAX$ complete isometrically. Paulsen \cite{PA92} introduced the following numerical quantity, \[\alpha(X):=\sup\Big\{\Big\|\sum_{i=1}^n A_i\otimes B_i\Big\|_{\ell_2\to\ell_2}\Big\}\] where the supremum is taken over all  $(A_1,\dots,A_n)\in B(\ell_2)$ and $(B_1,\dots,B_n)\in B(\ell_2)$ such that $\Big\|\sum_{i=1}^n\lambda_iA_i\Big\|_{\ell_2\to\ell_2}\leq 1,\Big\|\sum_{i=1}^n\mu_iB_i\Big\|_{_{\ell_2\to\ell_2}}\leq 1,$ for all $\|(\lambda_1,\dots,\lambda_n)\|_X\leq 1,$ $\|(\mu_1,\dots,\mu_n)\|_{X^*}\leq 1$ and $\text{dim}X=n.$ It is implicit in \cite{BAM95} that $\alpha(X)\geq \sqrt{\gamma(X)},$ which can be indeed proved easily by running the supremum as above, over matrices of which only first row is nonzero. Also, $MIN=MAX$ completely isometrically if and only if $\alpha(X)=1$ (see \cite{PA92}). Therefore, we have $\alpha(X)>1$ whenever $\gamma(X)>1.$
\section{A few elementary lemmas}\label{sec2}
	 We prove few elementary lemmas, the first of which connects the problem with Birkhoff-James orthogonality and norm parallelism.
	\begin{lem}\label{agag}
Let $1\leq p<\infty.$ Let $\Phi:\ell_p^2\to B(\mathcal \ell_2),$ defined by $\Phi((z_1,z_2)):=z_1T+z_2S$ be an isometry. Then, we have the following.
\begin{enumerate}
	\item[(i)] $\|T\|_{\ell_2\to\ell_2}=\|S\|_{\ell_2\to\ell_2}=1.$
	\item[(ii)] The linear map $\Phi_{new}:\ell_p^2\to B(\ell_2 \oplus_2\ell_2)$ defined by $\Phi_{new}((z_1,z_2)):=z_1T_{new}+z_2S_{new},$ where $T_{new}(\zeta_1\oplus\zeta_2):=T\zeta_2\oplus T^*\zeta_1$ and $S_{new}(\zeta_1\oplus\zeta_2):=S\zeta_2\oplus S^*\zeta_1$ are again isometries.
	\item[(iii)] For all $1\leq p<\infty,$ $T\perp_{BJ}S$ and $S\perp_{BJ}T.$
	\item[(iv)] For $p=1,$ $T||S$.
\end{enumerate}
	\end{lem}
\begin{proof}\begin{enumerate}
		\item[(i)] Putting $z_1=1$ and $z_2=0,$ we see that $\|T\|_{\ell_2\to\ell_2}=1.$ Similarly, $\|S\|_{\ell_2\to\ell_2}=1.$ 
		\item[(ii)] Just use that for all $z_1,z_2\in\mathbb C,$ $\|\overline{z_1}T^*+\overline{z_2}S^*\|_{\ell_2\to\ell_2}=\|z_1T+z_2S\|_{\ell_2\to\ell_2}=\|(z_1,z_2)\|_p.$
		\item[(iii)] Note that $\|T+zS\|_{\ell_2\to\ell_2}=\|(1,z)\|_p\geq 1=\|T\|_{\ell_2\to\ell_2}.$ Thus $T\perp_{BJ}S.$ Similarly, $S\perp_{BJ}T.$
		\item[(iv)] Note that for $z\in\mathbb T,$ $\|T+zS\|_{\ell_2\to\ell_2}=2=\|T\|_{\ell_2\to\ell_2}+\|S\|_{\ell_2\to\ell_2}$.
	\end{enumerate}
This completes the proof of the lemma.
\end{proof}

\begin{lem}\label{anesti}
	Let $c>0$ and $2<p<\infty.$ Then, there exists $\alpha>0$ such that we have \[(t^2+c^2)^p\geq t^{2p}+\alpha t^{2p-2}\] for all $t\geq 0.$ 
\end{lem}

\begin{proof}Put $p_0=\frac{p}{2}.$ Define the following function
	\[f(t):=(t^2+c^2)^p-t^{2p}-\alpha t^{2p-2}=(t^4+2t^2c^2+c^4)^{p_0}-t^{2p}-\alpha t^{2p-2}\] for $t>0,$ where $\alpha>0$ to be chosen later. By differentiating $f$ we have that \begin{equation}\label{peq1}
	\begin{split}
	f^{\prime}(t)&=p_0(t^4+2t^2c^2+c^4)^{p_0-1}(4t^3+4tc^2)-2pt^{2p-1}-(2p-2)\alpha t^{2p-3}\\
	&= 2tp(t^2+c^2)(t^2+c^2)^{p-2}-2pt^{p-1}-(2p-2)\alpha t^{2p-3}.
	\end{split}
	\end{equation} Note that we have the following inequality.
	\begin{equation}\label{peq2}
	\begin{split}
	2tp(t^2+c^2)(t^2+c^2)^{p-2}-2pt^{p-1}-(2p-2)\alpha t^{2p-3}&\\ \geq 2p(t^2+c^2)t^{2p-3}-2pt^{p-1}-(2p-2)\alpha t^{2p-3}&\\
	=(2pc^{2}-(2p-2)\alpha)t^{2p-3}.
	\end{split}
	\end{equation}In the first inequality above we have used that $t^2+c^2\geq t^2,$ for $t\geq 0.$ 

	Now, we choose $\alpha>0$ but small enough, such that $2pc^2-(2p-2)\alpha>0.$ Thus, in view of equation \eqref{peq1} and equation \eqref{peq2} have that $f^{\prime}(t)\geq 0$ for all $t>0$ and $f(0)>0.$ This completes the proof the lemma.
\end{proof}
\section{Proof of the theorems}\label{sec3}
\begin{proof}[Proof of Theorem \ref{noc}]Suppose not. Then there exists linear operators $T,S\in S_\infty$ such that \[\|z_1T+z_2S\|_\infty=|z_1|+|z_2|,\] for all $z_1,z_2\in\mathbb C.$ By $(ii)$ of Lemma \ref{agag}, first we can assume that $T$ and $S$ are self-adjoint. Observe that for any unitary operator $U:\ell_2\to\ell_2,$ we have 
	\[\|z_1T+z_2S\|_\infty=\|z_1UTU^*+z_2USU^*\|_\infty.\]
We have by Part $(i)$ of \Cref{agag}, $\|T\|_\infty=\|S\|_\infty=1.$ Therefore, we can further assume that $T$ is a diagonal matrix with all eigenvalues real, and if necessary applying again unitary transformation, we can further assume that $T$ has the following form. For some $k,l\in\mathbb N\cup\{0\}$ with $k+l\geq 1,$ $Te_i=e_i,$ $1\leq i\leq k$, $Te_i=-e_i$ for $k+1\leq i\leq k+l$ and $Te_j=\alpha_je_j,$ $|\alpha_j|<1$ for $k+l< j.$ The set $\{e_i:i\geq 1\}$ denotes the standard basis of $\ell_2.$ The case $k=0$ or $l=0$ corresponds to the cases that $T$ does not have $1$ or $-1$ respectively in the diagonal.

	Define the unitary operator $V:\ell_2\to\ell_2$ as $Ve_i=\text{sign}(\langle Te_i,e_i\rangle)e_i$ if $\langle Te_i,e_i\rangle\neq 0$ and $Ve_i=e_i$ if $Te_i=0$, $1\leq i<\infty.$ Now, note that $\|z_1TV+z_2SV\|_\infty=\|z_1T+z_2S\|_\infty,$ for all $z_1,z_2\in\mathbb C$ as $V$ is unitary. Therefore, if necessary again performing a unitary transformation, we may further assume that $T$ is a positive semi-definite diagonal matrix of the form $Te_i=e_i$ for $1\leq i\leq n$ and $Te_i=\alpha_{i}e_i$ for $i\geq n+1$ with $\alpha_i<1$ for all $i\geq n+1.$
	Clearly,
	$\mathbb{M}_T=\{\zeta\in\text{span}\{e_1,\dots,e_{n}\}:\|\zeta\|_2=1\}.$

	Using $(iv)$ of Lemma \ref{agag} and Theorem \ref{para}, we see that there exists $\zeta\in\mathbb{M}_T\cap \mathbb{M}_S$ such that $|\langle T\zeta, S\zeta\rangle|=\|T\|_\infty\|S\|_\infty.$ We readily have that for some nonzero $\alpha\in\mathbb C,$ $S\zeta=\alpha\zeta.$ Furthermore, taking norm and using the fact that $\zeta\in\mathbb M_S,$ we have $|\alpha|=1.$ Thus, $\alpha$ is a unimodulus eigenvalue of $S$.
	
	Consider the map $\ell_1^2\mapsto S_\infty,$ defined as $(z_1,z_2)\mapsto z_1T+z_2(\overline{\alpha} S).$ This is again an isometry which is justified as follows.
	\[\|z_1T+z_2\overline{\alpha} S\|_\infty=\|(z_1,\overline{\alpha} z_2)\|_1=\|(z_1,z_2)\|_1.\] In above we have used the fact that $|\alpha|=1.$
	
	Therefore, by replacing $S$ by $\overline{\alpha} S,$ we may assume that $S\zeta=\zeta.$ Thus, $\zeta$ is a common unit eigenvector of $T$ and $S$ with common eigenvalue $1.$
	
	Now, we choose an orthonormal basis $\{\eta_1,\eta_2\dots\}$ of $\ell_2$ as the following:

Let us fix $\eta_1=\zeta.$ By performing, Gram-Schimdt process, we choose $\eta_2,\dots,\eta_n$ in the subspace $\text{span}\{e_1,\dots,e_{n}\}$ such that $\{\eta_1,\dots,\eta_n\}$ is an orthonormal basis of $\text{span}\{e_1,\dots,e_{n}\}$. Now, we set $\eta_{i}=e_{i}$ for $i\geq n+1.$ It is clear that $\{\eta_1,\eta_2\dots\}$ is an orthonormal basis of $\ell_2.$
	
	Let us define $W:\ell_2\to\ell_2$ as $We_i=\eta_i$ for $i\geq 1.$ Thus, $W$ is a unitary operator and we have that 
	\[W^*TW=T\] and first column of $W^*SW$ is $(1,0,0,\dots).$ Denote $W^*TW=T_1$ and $W^*SW=S_1.$ Note that by taking adjoint \begin{equation}\label{bar}	\|\overline{z_1}T_1^*+\overline{z_2}S_1^*\|_\infty=\|z_1T_1+z_2S_1\|_\infty=\|(z_1,z_2)\|_1.
		\end{equation}
	Replacing $(z_1,z_2)$ by $(\overline{z_1},\overline{z_2})$ in equation \eqref{bar}, we obtain that the map $(z_1,z_2)\mapsto z_1 T_1+z_2S_1^*$ again induces an isometry from $\ell_1^2$ to $S_\infty.$ Note that the first row of $S_1^*$ is $(1,0,0,\dots).$ Now, by Part $(i)$ of Lemma \ref{agag}, we have $\|S_1^*\|_\infty=1.$ We denote $v:=S_1^*e_1.$ We must have $\|v\|_2\leq 1.$ Note that if we denote $(s_{ij})_{ i,j\geq 1}$ to be the matrix representation in the standard basis of $S_1^*$, then $v=(1,s_{21},s_{31},\dots).$ Therefore, as $\|v\|_2\leq 1$ we must have $s_{i1}=0$ for all $i\geq 2.$ Thus, the first column of $S_1^*$ is $(1,0,0,\dots)^t.$ We denote the operators $T_2$ and $S_2$ by eliminating first row and first column in the matrix representation of $T_1$ and $S_1^*$ respectively.
	Therefore, we obtain that
	\begin{equation}\label{eq1}
	|z_1|+|z_2|=\|z_1T_1+z_2S_1^*\|_\infty=\max\{|z_1+z_2|,\ \|z_1T_2+z_2S_2\|_\infty\},\  \text{for all} ,\ z_1,z_2\in\mathbb C.
	\end{equation}
	From equation \eqref{eq1}, we obtain that for all $z_1,z_2\in\mathbb C$ with $|z_1+z_2|<|z_1|+|z_2|$ we have \[\|z_1T_2+z_2S_2\|_\infty=|z_1|+|z_2|.\]
	By density and using continuity of norm, we obtain from above that the map defined by $(z_1,z_2)\mapsto z_1T_2+z_2S_2$ again defines an isometric embedding of $\ell_1^2$ into $S_\infty.$ We notice that $T_1$ is again a positive semi-definite diagonal matrix with $n-1$ number of $1$'s in the diagonal. Now, after finitely many iteration we will finally obtain an isometric embedding
	\[(z_1,z_2)\mapsto z_1\widetilde{T}+z_2\widetilde{S}\] where $\widetilde{T}$ is a positive semi-definite diagonal matrix with all diagonal elements strictly less than $1.$ But by Part $(i)$ of Lemma \ref{agag}, we have $\|\widetilde{T}\|_\infty=1,$ which is contradiction. This completes the proof of the theorem.
	\end{proof}

\begin{rem}
It is well known that $\ell_1^2(\mathbb R)$ is isometrically isomorphic to $\ell_\infty^2(\mathbb R).$ However, in complex case, $\ell_1^2$ is not isometrically isomorphic to $\ell_\infty^2.$ We have an isometric embedding of $\ell_1^2(\mathbb R)$ to $M_2(\mathbb R)$ defined as $(x_1,x_2)\mapsto x_1T+x_2S,$ where $T:=I_2,$ the identity matrix and $S:=e_{11}-e_{22}$ where $\{e_{ij}:1\leq i,j\leq 2\}$ is the standard basis of $M_2(\mathbb R).$ Therefore, it is reasonable to ask what makes the proof of Theorem \ref{noc} work? Interestingly, the reason for which the proof works for complex case, is due to the elementary fact that $\{(z_1,z_2):|z_1+z_2|<|z_1|+|z_2|\}$ is dense in $\mathbb C^2,$ which is not true for real case and this is the only point where the proof actually differs for the complex case. 
\end{rem}
	\begin{proof}[Proof of \Cref{forp}] Suppose not. Let $T$ be compact. As in the beginning of the proof of Theorem \ref{noc}, we may further assume that $T$ is a positive semi-definite diagonal matrix of the form $Te_i=e_i$ for $1\leq i\leq n$ and $Te_i=\alpha_{i}e_i$ for $i\geq n+1$ with $\alpha_i<1$ for all $i\geq n+1.$
		Clearly, we again have
		$\mathbb{M}_T=\{\zeta\in\text{span}\{e_1,\dots,e_{n}\}: \|\zeta\|_2=1\}.$

		By Part $(iii)$ of Lemma \ref{agag} and Theorem \ref{ort}, we obtain a vector $\zeta\in \mathbb{M}_T$ such that $\langle T\zeta, S\zeta\rangle=0.$ Note that for all $z_1,z_2\in\mathbb C,$ we have 
		\begin{equation}\label{eq2}
		(|z_1|^p+|z_2|^p)^{\frac{1}{p}}=\|z_1T+z_2S\|_{\ell_2\to\ell_2}\geq \|z_1T\zeta+z_2S\zeta\|_2=(|z_1|^2+c^2|z_2|^2)^{\frac{1}{2}},
		\end{equation} where $c:=\|S\zeta\|_2.$ We have the following cases.
		
		\textbf{Case 1} $c>0.$

		From equation \eqref{eq2}, using homogenity and Lemma \ref{anesti}, we have that for all $t> 0,$ and some constant $\alpha>0,$
		\[(t^p+1)^2\geq (t^2+c^2)^p\geq t^{2p}+\alpha t^{2p-2}.\] Therefore, \[2+\frac{1}{t^p}\geq \alpha t^{p-2}.\] Letting $t\to\infty,$ we get a contradiction.
	
	\textbf{Case 2} $c=0.$
	
	From the assumption, we obtain that $\|T\zeta\|_2=1=\|\zeta\|_2.$ This shows that $\zeta$ is a unit eigenvector of $T$ with eigenvalue $1.$ Also, $S\zeta=0.$ Hence, $\zeta$ is an eigenvector of $S$ with eigenvalue $0.$ Now, we choose an orthonormal basis $\{\eta_1,\eta_2\dots\}$ of $\ell_2$ as the following:
	
	Choose $\eta_1=\zeta.$ By Gram-Schimdt process we choose $\eta_2,\dots,\eta_n$ in the subspace $\text{span}\{e_1,\dots,e_{n}\}$ such that $\{\eta_1,\dots,\eta_n\}$ is an orthonormal basis of $\text{span}\{e_1,\dots,e_{n}\}$. Now we set $\eta_{i}=e_{i}$ for $i\geq n+1.$ It is clear that $\{\eta_1,\eta_2\dots\}$ is an orthonormal basis of $\ell_2.$
	
	Denote $W:\ell_2\to\ell_2$ defined as $We_i=\eta_i$ for $i\geq 1.$ Thus $W$ is a unitary operator on $\ell_2$ and we have that 
	\[W^*TW=T.\] Denote $W^*TW=T_1$ and $W^*SW=S_1.$

	Note that $(z_1,z_2)\mapsto z_1T_1+z_2S_1$ induces an isometry between $\ell_p^2$ to $B(\ell_2).$ Using matrix representation with respect to the standard basis we see that the first column of $S_1$ is entirely $0$ (since $S_1e_1=0.$) Now, consider a new isometric embedding  from $\ell_p^2$ to $B(\ell_2)$ defined by $(z_1,z_2)\mapsto z_1T_1^*+z_2S_1^*.$ The fact it is again an isometric embedding can be verified as follows 
	\[\|z_1T_1^*+z_2S_1^*\|_{\ell_2\to\ell_2}=\|\overline{z_1}T_1+\overline{z_2}S_1\|_{\ell_2\to\ell_2}=\|(\overline{z_1},\overline{z_2})\|_p=\|(z_1,z_2)\|_p.
	\]
	Now, $T_1^*$ is actually $T$ itself. Note that, now all the terms in the first row of $S_1^*$ becomes $0.$ Again $e_1$ is in $\mathbb M_T$ and $Te_1=e_1$ and $\langle Te_1,S_1^*e_1\rangle=0,$ since the first row of $S_1^*$ is zero. If $\|S_1^*e_1\|_2\neq 0,$ using Case 1, we shall obtain a contradiction. Therefore, we are forced to assume that $S_1^*e_1=0.$ Therefore, we obtain that the first column of $S_1^*$ is entirely zero.

	Denote $T_2:=T$ and $S_2:=S_1^*.$ Denote $T_3$ and $S_3$ to be operators by eliminating the first row and first column in the matrix representations of $T_2$ and $S_2$ respectively.
	Therefore, from the fact that $(z_1,z_2)\mapsto z_1T_2+z_2S_2$ is an isometric embedding from $\ell_p^2$ to $B(\ell_2)$, we have the identity
	\begin{equation}\label{eq3}
	(|z_1|^p+|z_2|^p)^{\frac{1}{p}}=\max\{|z_1|,\|z_1T_3+z_2S_3\|_{\ell_2\to\ell_2}\},\ \text{for all}\ z_1,z_2\in\mathbb C.
	\end{equation}
	Note that we always have for $z_2\neq 0,$ $\|(z_1,z_2)\|_p>|z_1|.$ Therefore, from equation \eqref{eq3}, we obtain 
	\begin{equation}\label{eq4}
	\|z_1T_3+z_2S_3\|_{\ell_2\to\ell_2}=(|z_1|^p+|z_2|^p)^{\frac{1}{p}}, \ \text{for all},\ z_2\neq 0.
	\end{equation}
	But by continuity of norm, we have from equation \eqref{eq4} \[\|z_1T_3+z_2S_3\|_{\ell_2\to\ell_2}=(|z_1|^p+|z_2|^p)^{\frac{1}{p}},\ \text{for all},\ z_1,z_2\in\mathbb C.\] Therefore, we obtain a new isometric embedding $(z_1,z_2)\mapsto z_1T_3+z_2S_3$ from $\ell_p^2$ to $B(\ell_2)$ where $T_3$ is again a positive semi-definite diagonal matrix but now there are $n-1$ number of $1$'s in the diagonal. Now, we can argue as in the proof of Theorem \ref{noc} to obtain a contradiction. This completes the proof of the theorem.
\end{proof}	
\begin{rem}
The proof of \Cref{forp} indicates that \Cref{forp} holds true for any two dimensional complex Banach space $(\mathbb C^2,\|.\|)$ with the following properties.
\begin{enumerate}
	\item For all $(z_1,z_2)\in\mathbb C^2,$ $\|(z_1,z_2)\|=\|(\overline{z_1},\overline{z_2})\|.$
	\item For all $0<c\leq 1,$ the operator $(z_1,z_2)\mapsto (z_1,cz_2)$ is not a contraction from $(\mathbb C^2,\|.\|)$ to $\ell_2^2.$
	\item $\{(z_1,z_2)\in\mathbb C^2:\|(z_1,z_2)\|>|z_1|\}$ is dense in $\mathbb C^2.$
	\item $\|(1,0)\|=\|(0,1)\|=1.$
	\item $\|(1,z)\|\geq 1$ and $\|(w,1)\|\geq 1$ for all $z,w\in\mathbb C.$
\end{enumerate}
Note that for all $2<p,q<\infty,$ the Banach spaces $(\mathbb C^2,\|.\|_{B_{p,q}})$ satisfies above properties since for any $2<p,q<\infty$, we can find $2<p_1,p_2<\infty$ such that $\|.\|_{\ell_{p_1}}\leq\|.\|_{B_{p,q}}\leq\|.\|_{\ell_{p_2}}.$ Hence, there is no isometric embedding of  $(\mathbb C^2,\|.\|_{B_{p,q}})$ into the set of all compact operators for $2<p,q<\infty.$ However, $(\mathbb C^2,\|.\|_{B_{1,2}})$ embeds isometrically into $S_\infty^2$ via the embedding \[(z_1,z_2)\mapsto z_1I_2+z_2e_{12}.\] We conjecture that  $(\mathbb C^2,\|.\|_{B_{p,q}})$ embeds isometrically into $S_\infty$ if and only if $(p,q)=(2,2),(1,2)$ or $(2,1).$
\end{rem}
\begin{proof}[Proof of \Cref{ncprecha}]If $S_1^m$ embeds isometrically into $S_p$, for some $m\geq 2,$ we can observe $\ell_1^2$ embeds into $S_p$ isometrically. Let $(z_1,z_2)\mapsto z_1T+z_2S$ be an isometry from $\ell_1^2$ into $S_p.$ Clearly, we have $T||S.$ Therefore, by Theorem \ref{FM1}, $T$ and $S$ are linearly dependent, which is absurd.
This completes the proof of the theorem.
\end{proof}
From now on, we shall always assume that $\Omega$ is a Reinhardt domain in $\mathbb{C}^2$ which is a unit ball with respect to some norm.
\begin{proof}[Proof of \Cref{auoss}]
	We first prove the following general fact. Suppose $\Omega$ is a unit ball with respect to some norm in $\mathbb{C}^2$, then $(\mathbb{C}^2,\|\cdot\|_{\Omega})$ has Property P if and only if $(\mathbb{R}^2,\|\cdot\|_{|\Omega|})$ has Property P.

		First note for a complex matrix $A=(a_{ij})_{i,j=1}^2\geq 0$ and $A:(\mathbb{C}^2,\|\cdot\|_{\Omega})\mapsto(\mathbb{C}^2,\|\cdot\|_{\Omega})^*$, we have that $\|A\|_{op}=\sup_{(z_1,z_2)\in\Omega}\sum_{i,j=1}^2a_{ij}z_i\bar{z}_{j}.$ We give a quick proof.
	Observe that, using duality, we have $\|A\|_{op}=\sup_{v\in\Omega,w\in\Omega}|\langle Av,w\rangle|.$
		By using the fact that if $A\geq 0$, we can always find a positive square root, say $B$ of $A$, we obtain the following
		\begin{align}\label{f}
		\sup_{v\in\Omega,w\in\Omega}|\langle Av,w\rangle|=\sup_{v\in\Omega,w\in\Omega}\langle Bv,Bw\rangle
		\leq\sup_{v\in\Omega,w\in\Omega}\|Bv\|_2\|Bw\|_2
		=\sup_{v\in\Omega}\|Bv\|_2^2.\nonumber
		\end{align}
		To complete the proof, we notice that one can always take $v=w$ in above.
	
		Observe that since $(\mathbb{C}^2,\|\cdot\|_{\Omega})$ is a Banach space with the Reinhardt unit ball $\Omega,$ the unit ball of $(\mathbb{C}^2,\|\cdot\|_{\Omega})^*$ is again Reinhardt. Also, by using Reinhardt property of the unit ball, we have $\|(z_1,z_2)\|_{\Omega}=\|(|z_1|,|z_2|)\|_{|\Omega|}\ , z_1,z_2\in\Omega.$	Therefore, for any complex matrix $A\geq 0$ we have the following equality \[ \|A\|_{((\mathbb{C}^2,\|\cdot\|_{\Omega})\rightarrow(\mathbb{C}^2,\|\cdot\|_{\Omega})^*)}=\|\mathbf A^{+}\|_{((\mathbb{C}^2,\|\cdot\|_{\Omega})\rightarrow(\mathbb{C}^2,\|\cdot\|_{\Omega})^*)}. \] The same is true for the real case. For $A\geq 0$ real matrix
		\[ \|A\|_{((\mathbb{R}^2,\|\cdot\|_{|\Omega|})\rightarrow(\mathbb{R}^2,\|\cdot\|_{|\Omega|})^*)}=\|\mathbf A^{+}\|_{((\mathbb{R}^2,\|\cdot\|_{|\Omega|})\rightarrow(\mathbb{R}^2,\|\cdot\|_{|\Omega|})^*)}. \]

	Thus, we observe for any complex matrix $A\geq 0$, we have \[ \|\mathbf A^{+}\|_{((\mathbb{C}^2,\|\cdot\|_{\Omega})\rightarrow(\mathbb{C}^2,\|\cdot\|_{\Omega})^*)}=\|\mathbf A^{+}\|_{((\mathbb{R}^2,\|\cdot\|_{|\Omega|})\rightarrow(\mathbb{R}^2,\|\cdot\|_{|\Omega|})^*)}. \]
All these above statements are true, because one can see
		\begin{align*}
		\sup_{(z_1,z_2)\in\Omega}(a_{11}|z_1|^2+2\text{Re}(a_{12}z_1\overline{z_2})+a_{22}|z_2|^2)
		&=\sup_{(|z_1|,|z_2|)\in\Omega}(a_{11}|z_1|^2+2|a_{12}||z_1||z_2|+a_{22}|z_2|^2)\\
		&=\sup_{\|(x,y)\|_{|\Omega|}\leq 1}(a_{11}x^2+2|a_{12}||x||y|+a_{22}y^2)\\
		&=\sup_{\|(x,y)\|_{|\Omega|}\leq 1}(a_{11}x^2+2|a_{12}|xy+a_{22}y^2).
	\end{align*}
		To this end let us observe
		\begin{align*}
		& =  \sup\{\langle A,B\rangle :\|A
		\|_{((\mathbb{C}^2,\|\cdot\|_{\Omega})\rightarrow(\mathbb{C}^2,\|\cdot\|_{\Omega})^*)}\leq 1,\|B\|_{((\mathbb{C}^2,\|\cdot\|_{\Omega}^*)\rightarrow(\mathbb{C}^2,\|\cdot\|_{\Omega}))}\leq 1,A\geq 0,B\geq 0\}\\
		&\leq\sup\{\langle \mathbf A^{+},\mathbf B^{+}\rangle :\|A
		\|_{((\mathbb{C}^2,\|\cdot\|_{\Omega})\rightarrow(\mathbb{C}^2,\|\cdot\|_{\Omega})^*)}\leq 1,\|B\|_{((\mathbb{C}^2,\|\|_{\Omega}^*)\rightarrow(\mathbb{C}^2,\|\cdot\|_{\Omega}))}\leq 1,A\geq 0,B\geq 0\}.
		\end{align*}
		The last inequality is just the triangle inequality and it is clear that it is actually an equality. By a very similar argument as above, we also have 
			\begin{align*}
		& \sup\{\langle A,B\rangle :\|A
		\|_{((\mathbb{R}^2,\|\cdot\|_{\Omega})\rightarrow(\mathbb{R}^2,\|\cdot\|_{\Omega})^*)}\leq 1,\|B\|_{((\mathbb{R}^2,\|\cdot\|_{\Omega}^*)\rightarrow(\mathbb{R}^2,\|\cdot\|_{\Omega}))}\leq 1,A\geq 0,B\geq 0\}\\
		&\leq\sup\{\langle \mathbf A^{+},\mathbf B^{+}\rangle :\|A
		\|_{((\mathbb{R}^2,\|\cdot\|_{\Omega})\rightarrow(\mathbb{R}^2,\|\cdot\|_{\Omega})^*)}\leq 1,\|B\|_{((\mathbb{R}^2,\|\|_{\Omega}^*)\rightarrow(\mathbb{R}^2,\|\cdot\|_{\Omega}))}\leq 1,A\geq 0,B\geq 0\}.
		\end{align*}
	One readily sees that $\gamma((\mathbb{R}^2,\|\cdot\|_{|\Omega|}))=\gamma((\mathbb{C}^2,\|\cdot\|_{\Omega}))$.

	Now, by \Cref{twoo}, one can easily observe that if $\Omega$ is of the form $\{(z_1,z_2):|z_1|^p+|z_2|^q<1\}$, where $p$ and $q$ are real numbers bigger than or equal to one and at least one of them is strictly bigger than one, then $(\mathbb{C}^2,\|\cdot\|_{\Omega})$ cannot have Property P, as $\overline{|\Omega|}=\{(x,y):|x|^p+|y|^q\leq 1,x,y\in\mathbb{R}\}$ has more than four extreme points but the closed unit ball of the real Banach space $\ell_\infty^2(\mathbb R)$ has exactly four extreme points. 
	This completes the proof of the theorem.
	\end{proof}
\begin{rem} 
It is easy to see that Property P is preserved under duality. Combining this and the characterization of Thullen for Reinhardt domains in $\mathbb{C}^2$, we get a very large class of two dimensional complex Banach spaces, which can be endowed with different operator space structures. 
	\end{rem}

\textbf{Acknowledgement:} The author is very grateful to Guixiang Hong, Gadadhar Misra and Md. Ramiz Reza and  for many helps related to this current work. The author acknowledges Key-subsidy postdoctoral fellowship supported by Wuhan University.

\end{document}